\documentclass[a4 paper,11pt]{article}
\usepackage{stmaryrd}
\usepackage{amsfonts}
\usepackage{bbm}
\usepackage{amscd}
\usepackage{mathrsfs}
\usepackage{latexsym,amssymb,amsmath,amscd,amscd,amsthm,amsxtra,xypic}
\usepackage[dvips]{graphicx}
\usepackage[utf8]{inputenc}
\usepackage[T1]{fontenc}
\usepackage{enumerate}
\usepackage{lmodern}
\usepackage{amssymb}
\usepackage[all]{xy}
\usepackage{ifpdf}
\ifpdf
\usepackage[colorlinks=true,linkcolor=blue,final,backref=page,hyperindex,citecolor=red]{hyperref}
\else
\usepackage[colorlinks,final,backref=page,hyperindex,hypertex]{hyperref}
\fi

\usepackage{nicefrac,mathtools,enumitem}
\usepackage{microtype}
\usepackage{amsfonts}
\usepackage{amssymb}
\usepackage{mathrsfs}
\usepackage{amsmath}
\usepackage{amsthm}
\usepackage{enumerate}
\usepackage{indentfirst}
\usepackage{amsfonts}
\usepackage{amssymb}
\usepackage{mathrsfs}
\usepackage{amsmath}
\usepackage{amsthm}
\usepackage{enumerate}
\usepackage{cite}
\usepackage{mathrsfs}
\usepackage{geometry}
\allowdisplaybreaks

\newtheorem{thm}{Theorem}[section]

\newtheorem{pro}[thm]{Proposition}

\newtheorem{rmk}[thm]{Remark}
\newtheorem{defi}[thm]{Definition}

\numberwithin{equation}{section}

\newcommand{\be }{\begin{equation}}
\newcommand{\ee }{\end{equation}}

\newcommand{\br}[1]{   [ \cdot,    \cdot  ]   }

\newcommand {\emptycomment}[1]{}

\def\<{\langle}
\def\>{\rangle}

\begin{document}

\title{\bf  The solenoidal Heisenberg Virasoro algebra and its simple weight modules}

\author{\bf B. Agrebaoui, W. Mhiri}
\author{{ Boujemaa Agrebaoui$^{1}$
 \footnote{E-mail: b.agreba@fss.rnu.tn}~ and \  Walid Mhiri$^{1}$ \footnote{E-mail: mhiriw1@gmail.com}
\
}\\
 \\
{\small 1.~University of Sfax, Faculty of Sciences Sfax,  BP
1171, 3038 Sfax, Tunisia}}

\date{}

\maketitle

\begin{abstract}
Let $A_n=\mathbb{C}[t_i^{\pm1},~1\leq i\leq n]$ and $\mathbf{W}(n)_\mu=A_nd_\mu$ the solenoidal Lie algebra introduced by Y.Billig and V.Futorny in \cite{BiFu2}, where $\mu=(\mu_1,\ldots,\mu_n)\in\mathbb{C}^n$ is a generic vector and $$d_\mu=\sum_{i=1}^n\mu_it_i\frac{\partial}{\partial t_i}.$$ We consider the semi-direct product Lie algebra $\mathbf{WA}(n)_\mu:=\mathbf{W}(n)_\mu\ltimes A_n$.

In the first part, We prove that $\mathbf{WA}(n)_\mu$ has a unique three-dimensional  universal central extension. In fact we construct a higher rank Heisenberg-Virasoro algebra (see \cite{LiuGuo, LdZ}). It will be denoted by $\mathbf{HVir}(n)_\mu$ and it will be called the solenoidal  Heisenberg-Virasoro algebra. Then we will study Harish-Chandra modules of $\mathbf{HVir}(n)_\mu$ following \cite{LiuGuo}. We will obtain two classes of Harich-Chandra modules: generalized highest weight modules(\textbf{GHW} modules) and intermediate series modules. Our results are particular cases of \cite{LiuGuo}.
In the end, we will construct  $\mathbf{HVir}(n)_\mu$ Verma modules  using the lexicographic order on $\mathbb{Z}^{n}$. In particular we give examples of irreducible weight modules  which have infinite dimensional weight spaces.
\end{abstract}

\textbf{Key words}: Heisenberg-Virasoro algebra, solenoidal algebra, solenoidal Heisenberg Virasoro algebra, central extension, Harish-Chandra modules, cuspidal modules

\textbf{Mathematics Subject Classification} (2010): 17B10,17B20,17B68,17B86.

\numberwithin{equation}{section}

\tableofcontents

\section{Introduction}
The Heisenberg-Virasoro algebra $\mathbf{HVir}$ was first introduced in \cite{ACKP}, where highest weight modules were studied and a
determinant formula for the Shapovalov form on Verma modules was obtained. In \cite{LuZhao} (see also\cite{LiuJiang},\cite{DLD}), Lu and Zhao classified the irreducible Harish-Chandra modules over $\mathbf{HVir}$, which turn out to be modules of intermediate series and highest/lowest weight modules. Whittaker modules for $\mathbf{HVir}$ were studied by \cite{LWZ}. Recently, a large class of irreducible non-weight modules were constructed in \cite{CG}.
The generalized Heisenberg Virasoro algebras are generalization of the Heisenberg-Virasoro algebras where the grading by $\mathbb{Z}$ is replaced by an additive subgroup $G$ of $\mathbb{C}$. Their representation theory was considered by several authors, see for example \cite{LdZ,SJS}.

Recently in \cite{BiFu, BiFu1}, Y. Billig and V. Futorny  study  weight modules of finite weight spaces of the Lie algebra $\mathbf{W}(n)$ of vector fields on the torus. They prove that such modules are highest modules or quotients of modules of tensor fields. In \cite{BiFu2}, they introduced so called solenoidal Lie algebra $\mathbf{W}(n)_\mu:=A_nd_\mu$ as a bridge between the Lie algebra $\mathbf{W}(1)$ and the Lie algebra $\mathbf{W}(n)$
where $\mu=(\mu_1,\ldots\mu_n)$ is a generic element in $\mathbb{C}^n$ and $d_\mu=\displaystyle\sum_{i=1}^n\mu_it_i\frac{\partial}{\partial t_i}$.  Then they give a classification of the simple cuspidal $\mathbf{W}(n)_{\mu}$-modules. In a forthcoming paper (see \cite{AM}), we compute the second cohomology space $H^2(\mathbf{W}(n)_{\mu},\mathbb{C})$. The universal central extension of $\mathbf{W}(n)_{\mu}$ is a new generalization of the Virasoro  algebra, denoted $\mathbf{Vir}(n)_\mu$ and is called the solenoidal-Virasoro algebra.
Then we give a complete classification of its Harish-Chandre modules.

In this paper we consider the semi-direct product $\mathbf{WA}(n)_\mu:=\mathbf{W}(n)_\mu\ltimes A_n$, the analogue of the Lie algebra $\mathbf{WA}(1)=\mathbf{W}(1)\ltimes A_1$ in the case $n=1$. The \textbf{first section} of this paper contains our main result given by Theorem \ref{Ext2}. We compute three generating $2$-cocycles and then we classify the universal central extension of $\mathbf{WA}(n)_\mu$. The obtained three-dimensional central extension of $\mathbf{WA}(n)_\mu$ is called the solenoidal Heisenberg-Virasoro  algebra and is denoted by $\mathbf{H}\mathbf{Vir}(n)_\mu$.  In the \textbf{second section}, we study Harish-Chandra modules over $\mathbf{H}\mathbf{Vir}(n)_\mu.$
In \cite{LiuGuo}, G.Liu and X.Guo give the definition of
generalized Heisenberg-Virasoro algebras $\mathbf{HVir}[G]$ where $G$ is an additive subgroup of $\mathbb{C}$. When $G\simeq\mathbb{Z}^n, \mathbf{HVir}[G]$
is called rank $n$ Heisenberg-Virasoro algebra. Our algebra $\mathbf{H}\mathbf{Vir}(n)_\mu$ is an example of  rank $n$ Heisenberg-Virasoro algebra.

In the second section, following \cite{LiuGuo}, we classify Harish-Chandra modules of $\mathbf{H}\mathbf{Vir}(n)_\mu$. We obtain tow kinds of modules, generalized highest weight modules (\textbf{GHW} modules) or intermediate series modules. For $n=1$, we obtain the classification results for the classical Heisenberg-Virasoro algebra given by R. L\"{u} and K. Zhao (see  \cite{LuZhao}).

In the\textbf{ third section}, we introduce a triangular decomposition of $\mathbf{H}\mathbf{Vir}(n)_\mu$ using the lexicographic order on $\mathbb{Z}^n$, then we define Verma modules and anti-Verma modules. As the usual highest weight theory, we obtain irreducible highest weight modules and irreducible lowest weight modules of $\mathbf{H}\mathbf{Vir}(n)_{\mu}$ by taking respectively quotients of  Verma modules and anti-Verma modules. In the end, we provide that these modules have infinite dimensional weight spaces.

\section{The solenoidal Heisenberg-Virasoro algebra $\mathbf{HVir}(n)_{\mu}$}
Let $A_n=\mathbb{C}[t_{i}^{\pm1},~1\leq i\leq n]$ be the algebra of Laurent polynomials and let  $\mu=(\mu_{1},\ldots,\mu_{n})\in\mathbb{C}^{n}$ generic, that is, for all $\alpha=(\alpha_1,\ldots,\alpha_n)\in \mathbb{Z}^n$ , $\mu\cdot\alpha:=\displaystyle\sum_{i=1}^n\mu_i\alpha_i\neq 0$. Let $d_{\mu}:=\displaystyle\sum_{i=1}^{n}\mu_{i}D_{t_{i}},$ where $D_{t_{i}}=t_{i}\frac{\partial}{\partial t_{i}}$. Y. Billig and V. Futorny \cite{BiFu2}, introduced the  solenoidal-Witt Lie algebra  $\mathbf{W}(n)_{\mu}:=A_nd_{\mu}$ as the Lie subalgebra of the Lie algebra $\mathbf{W}(n)=Der(A_n).$
 Let $$\Gamma_\mu=\{\mu\cdot\alpha;\alpha\in\mathbb{Z}^{n}\}.$$ It is the image of $\mathbb{Z}^{n}$ by the map : $$\begin{array}{cc}
                                                                     \sigma_\mu:& \mathbb{Z}^{n}\longrightarrow  \mathbb{C} \\
                                                                      &\alpha\mapsto  \mu\cdot\alpha
                                                                      \end{array}$$
$\Gamma_\mu$ is a subgroup of $(\mathbb{C},+).$
A canonical basis of $\mathbf{W}(n)_\mu$ is given by: $$\{e_{\mu\cdot\alpha}:=t^{\alpha}d_{\mu},\mu\cdot\alpha\in\Gamma_{\mu}\}.$$
 The commutators of the $e_{\mu\cdot\alpha}$  are given by:
\begin{equation}\label{soloWitt}[e_{\mu\cdot\alpha},e_{\mu\cdot\beta}]=\mu\cdot(\beta-\alpha)e_{\mu\cdot(\alpha+\beta)},~~\mu\cdot\alpha,\mu\cdot\beta\in \Gamma_{\mu}.\end{equation}
In the case of $n=1$,  we take $\mu \in \mathbb{C}^*$ then $\Gamma_{\mu}=\mu\mathbb{Z}$ and $\mathbf{W}(n)_\mu$ is isomorphic to $\mathbf{W}(1)$ by taking $d_m\rightarrow \gamma d_m$ where $\gamma$ is the square root of $\mu$ . In particular if $\mu =1$ we obtain the classical Witt algebra  $\mathbf{W}(1)$.

In the recent paper (see \cite{AM}) , we study the central extension of the solenoidal Lie algebra $\mathbf{W}(n)_\mu$ introduced by Y. Billig and V.Futorny (see \cite{BiFu2}), we obtain an analogue of the Virasoro algebra and we  called it the solenoidal Virasoro algebra and we denoted it by $\mathbf{Vir}(n)_\mu$. Then we give a classification of Harish-Chandra modules over $\mathbf{Vir}(n)_\mu$. Also, we construct $\mathbf{Vir}(n)_\mu$-modules with infinite dimensional weight spaces by using the lexicographic order on $\mathbb{Z}^{n}$.

In this paper we consider the Lie algebra $\mathbf{WA}(n)_\mu:=\mathbf{W}(n)_{\mu}\ltimes A_n$. Its canonical basis is:
$$\{e_{\mu\cdot\alpha}=t^{\alpha}d_{\mu}, h_\alpha=t^{\alpha},\mu\cdot\alpha\in\Gamma_\mu,~\alpha\in\mathbb{Z}^{n}\}$$
Its Lie structure generated by the following brackets:
$$[e_{\mu\cdot\alpha},e_{\mu\cdot\beta}]=\mu\cdot(\beta-\alpha)e_{\mu\cdot(\alpha+\beta)}.$$
$$[h_{\alpha},h_{\beta}]=0.$$
$$[e_{\mu\cdot\alpha},h_\beta]=(\mu\cdot\beta)h_{\alpha+\beta}.$$

The main purpose of this paper is to compute central extensions of the algbera $\mathbf{WA}(n)_\mu$.

The following theorem is a generalization to multidimensional case of Theorem 3 and Proposition 3 in  \cite{OvRo}  where the extension of the Lie algebra $Vect(S^1)$ of vector fields on the circle by modules of tensor densities $\mathcal{F}_\lambda$ is study.
\begin{thm}\label{Ext2} The second cohomology space $H^2(\mathbf{WA}(n)_{\mu},\mathbb{C})$ is three dimensional and it is generated by the following $2$-cocycles
$C_{\mu,1},~C_{\mu,2},~C_{\mu,3}:\mathbf{WA}(n)_{\mu}\times\mathbf{WA}(n)_{\mu} \longrightarrow \mathbb{C}$ defined by:

\begin{equation}\label{cocyc1}\left\{\begin{array}{ll} C_{\mu,1}(e_{\mu\cdot\alpha},e_{\mu\cdot\beta}):=\delta_{\alpha,-\beta}\frac{(\mu\cdot\alpha)^{3}-(\mu\cdot\alpha)}{12}c_{\mu,1}\\ 0\hbox{ otherwise}\end{array}\right.,\end{equation}

\begin{equation}\label{cocycl3}\left\{\begin{array}{ll}
C_{\mu,2}(e_{\mu\cdot\alpha},h_\beta):=\delta_{\alpha,-\beta}((\mu\cdot\alpha)^{2}-(\mu\cdot\alpha))c_{\mu,2}\\  0\hbox{ otherwise}\end{array}\right.,
\end{equation}
\begin{equation}\label{cocycl2}\left\{\begin{array}{ll}
C_{\mu,3}(h_\alpha,h_\beta):=\delta_{\alpha,-\beta}\frac{(\mu\cdot\alpha)}{3}c_{\mu,3} \\ 0\hbox{ otherwise}\end{array}\right.,\end{equation}

\end{thm}
\begin{proof}
The fact that the $2$-cochains $$C_{\mu,1},~C_{\mu,2},~C_{\mu,3}:\mathbf{WA}(n)_{\mu}\times\mathbf{WA}(n)_{\mu} \longrightarrow \mathbb{C}$$ are $2$-cocycles
 is a straight forward computations using the $2$-cocycle condition:
\begin{equation}\label{CC1}C_{\mu,i}([X,Y],Z)+C_{\mu,i}([Y,Z],X)+C_{\mu,i}([Z,X],Y)=0,\end{equation}
For $i=1,2,3;~X,Y,Z\in \mathbf{WA}(n)_{\mu}$.

 Let us now prove the unicity of the $2$-cocycles  $C_{\mu,1},~C_{\mu,2},~C_{\mu,3}$.

 Denote $X_{\alpha,1}=e_{\mu\cdot\alpha}$ and $X_{\alpha,2}= h_{\alpha}$. The first step, we prove that for $i\in\{1,2,3\}$ and $j,k\in\{1,2\}$ each cocycle has the following form: $$C_{\mu,i}(X_{\alpha,j},X_{\beta,k})=\delta_{i,j+k-1}\delta_{\alpha,-\beta}\theta_i(\mu\cdot\alpha)c_{\mu,i}, \hbox{ for all }\alpha,\beta\in \mathbb{Z}^n.$$
The second step, we apply known results on functional equations (see \cite{KaRiSa}, \cite{Az}) to give the final expressions.

Take $X=X_{\alpha,j}, Y=X_{\beta,k}$ and $Z=X_{\gamma,l}$.
Since condition (\ref{CC1}) is cyclic in $X, Y, Z$, it suffices to take $(j,k,l)\in \{(1,1,1),(1,1,2),(1,2,2)\}$ corresponding respectively to $\{C_{\mu,1},~C_{\mu,2},~C_{\mu,3}\}$ since the left hand side in condition (\ref{CC1}) is equal to zero for the other possibilities.

 Let us start by proving the unicity of $C_{\mu,1}$. So we take $(j,k,l)=(1,1,1)$, that is $(X_{\alpha,1},X_{\beta,1},X_{\gamma,1})=(e_{\mu\cdot\alpha},e_{\mu\cdot\beta},e_{\mu\cdot\gamma})$.
Assume that there exists $\Psi_1:\Gamma_{\mu}\times\Gamma_{\mu}\rightarrow \mathbb{C}$
such that:
\begin{equation}\label{bracketsolVir}[e_{\mu\cdot\alpha},e_{\mu\cdot\beta}]_{HVir_{\mu}}=(\mu\cdot\beta-\mu\cdot\alpha)e_{\mu\cdot(\alpha+\beta)}+\Psi_1(\mu\cdot\alpha,\mu\cdot\beta)c_{\mu,1}.\end{equation}

The function $\Psi_1(\mu\cdot\alpha,\mu\cdot\beta)$ can not be chosen arbitrary because of the anti-commutativity of the bracket and of the Jacobi identity.
We observe from (\ref{bracketsolVir}) that if we put: $$e'_{\mu\cdot0}=e_{\mu\cdot0},e'_{\mu\cdot\alpha}=e_{\mu\cdot\alpha}+\frac{\Psi_1(0,\mu\cdot\alpha)}{\mu\cdot\alpha}c_{\mu,1},~~ (\alpha\neq\overrightarrow{0}),$$
then we will have $$[e'_{\mu\cdot0},e'_{\mu\cdot\alpha}]_{HVir_{\mu}}=(\mu\cdot\alpha)e'_{\mu\cdot\alpha} \hbox{ for all }\mu\cdot\alpha\in\Gamma_\mu.$$
This transformation is merely a change of basis and we can drop the prime and say that:
\begin{equation}\label{ident2}[e_{\mu\cdot0},e_{\mu\cdot\alpha}]_{HVir_{\mu}}=(\mu\cdot\alpha)e_{\mu\cdot\alpha}\hbox{ for all }\mu\cdot\alpha\in\Gamma_\mu.\end{equation}
From the Jacobi identity for $e_{\mu\cdot0},e_{\mu\cdot\alpha},e_{\mu\cdot\beta}$ we get
\begin{equation}\label{ident3}[e_{\mu\cdot0},[e_{\mu\cdot\beta},e_{\mu\cdot\alpha}]_{HVir_{\mu}}]_{HVir_{\mu}}=\mu\cdot(\beta+\alpha)[e_{\mu\cdot\beta},e_{\mu\cdot\alpha}]_{HVir_{\mu}}\end{equation}
Substituting (\ref{bracketsolVir}) in (\ref{ident3}) and using (\ref{ident2}) we get: $$\mu\cdot(\alpha+\beta)\Psi_1(\mu\cdot\alpha,\mu\cdot\beta)c_{\mu,1}=0.$$
But this is equivalent to $\alpha+\beta =\overrightarrow{0}$ or $\Psi_1(\mu\cdot\alpha,\mu\cdot\beta)=0$. Then $\Psi_1$ has the following form: \begin{equation} \label{exptheta}\Psi_1(\mu\cdot\alpha,\mu\cdot\beta)=\delta_{\alpha,-\beta}\theta_1(\mu\cdot\alpha)\end{equation}
 where $\theta_1$ is a function from $\Gamma_{\mu}$ to $\mathbb{C}$.

The Lie bracket  (\ref{bracketsolVir}) becomes:
\begin{equation}\label{ident4}[e_{\mu\cdot\alpha},e_{\mu\cdot\beta}]_{HVir_{\mu}}=(\mu\cdot\beta-\mu\cdot\alpha)e_{\mu\cdot(\alpha+\beta)}+\delta_{\alpha,-\beta}\theta_1(\mu\cdot\alpha)c_{\mu,1},~\mu\cdot\alpha,\mu\cdot\beta\in~\Gamma_\mu.\end{equation}
By antisymmetry of the bracket, we deduce that $\theta_1$ is an odd function ($\theta_1(\mu\cdot\alpha)=-\theta_1(-\mu\cdot\alpha)$) and by bilinearity of the bracket, we deduce that $\theta_1$ is additive. So, $\theta_1$ is a group morphism from $(\Gamma_{\mu},+)$ to $(\mathbb{C},+)$.

We now work out the  $2$-cocycle condition (\ref{CC1}) on $C_{\mu,1}$ for $e_{\mu\cdot\gamma},e_{\mu\cdot\alpha},e_{\mu\cdot\beta}$. If $\gamma+\beta+\alpha\neq \overrightarrow{0}$ then (\ref{CC1}) is satisfied. If $\gamma+\beta+\alpha=\overrightarrow{0}$, using (\ref{ident4}) and the the fact that $\theta_1$ is odd, we get from (\ref{CC1}) the following equation:
\begin{equation}\label{ident5}\mu\cdot(\alpha-\beta)\theta_1(\mu\cdot(\alpha+\beta))-\mu\cdot(2\beta+\alpha)\theta_1(\mu\cdot\alpha)+\mu\cdot(\beta+2\alpha)\theta_1(\mu\cdot\beta)=0\end{equation}
where $\theta_1$ is a continuous function.
Substituting $\beta$ by $-\beta$ in (\ref{ident5}) we obtain the following equation:
\begin{equation}\label{ident6}\mu\cdot(\alpha+\beta)\theta_1(\mu\cdot(\alpha-\beta))-\mu\cdot(\alpha-2\beta)\theta_1(\mu\cdot\alpha)-\mu\cdot(2\alpha-\beta)\theta_1(\mu\cdot\beta)=0\end{equation}
by adding (\ref{ident5}) and (\ref{ident6}) we get:
\begin{equation}\label{ident8}(\mu\cdot\alpha)[\theta_1(\mu\cdot(\alpha+\beta))+\theta_1(\mu\cdot(\alpha-\beta))-2\theta_1(\mu\cdot\alpha)]=(\mu\cdot\beta)[\theta_1(\mu\cdot(\alpha+\beta))+\theta_1(\mu\cdot(\beta-\alpha))-2\theta_1(\mu\cdot\beta)]\end{equation}
Let us denoted $x:=\mu\cdot\alpha$ and $y:=\mu\cdot\beta$ and replace them in (\ref{ident8}) we will obtain:
\begin{equation}\label{ident9}x[\theta_1(x+y)+\theta_1(x-y)-2\theta_1(x)]=y[\theta_1(x+y)-\theta_1(x-y)-2\theta_1(y)]\end{equation}
But (\ref{ident9}) is equivalent to the following equation:
\begin{equation}\label{ident10}2x\theta_1(x)-2y\theta_1(y)=(x-y)\theta_1(x+y)+(x+y)\theta_1(x-y).\end{equation}

Using results on functional equations by PL.Kannappan,T.Riedel and P.K.Sahoo (see \cite{KaRiSa}), the equation (\ref{ident10})
has the following general solution: $$\theta_1(x)=ax^{3}+A(x)$$ where $A:\mathbb{C}\mapsto\mathbb{C}$ is an additive function. Since we work with continuous function $\theta_1$, then $A$ will be continuous and additive function, and so it is a linear function $A(x)= bx, b\in \mathbb{C}$.

 Finally, $\theta_1(x)=ax^{3}+bx$ where $a,b~\in~\mathbb{C}$ and for $x=\mu\cdot\alpha$ we have: $$\theta_1(\mu\cdot\alpha)= a(\mu\cdot\alpha)^{3}+b(\mu\cdot\alpha) .$$
The $2$-cocycle $\theta_1$ is non trivial if and only if  $a\neq0$ while $b$ can be chosen arbitrary. By the convention taken in Virasoro $2$-cocycle ( $n=1$ ), the choice $a=-b=\frac{1}{12}$ and the generating $2$-cocycle becomes:
\begin{equation}\label{SVirC}C_{\mu,1}(e_{\mu\cdot\alpha},e_{\mu\cdot\beta})=\delta_{\alpha,-\beta}\theta_1(\mu\cdot\alpha)c_{\mu,1}=\frac{(\mu\cdot\alpha)^{3}-\mu\cdot\alpha}{12}\delta_{\alpha,-\beta}c_{\mu,1}.\end{equation}

For the unicity of the $2$-cocycle $C_{\mu,2}$, we take $(j,k,l)=(1,1,2)$.
Assume that there exists $\Psi_2:\Gamma_\mu\times\Gamma_\mu\rightarrow \mathbb{C}$
such that:
\begin{equation}\label{crochetSVir2}[e_{\mu\cdot\alpha},h_{\beta}]_{HVir_{\mu}}=(\mu\cdot\beta)h_{\alpha+\beta}+\Psi_2(\mu\cdot\alpha,\mu\cdot\beta)c_{\mu,2}.\end{equation}

The function $\Psi_2(\mu\cdot\alpha,\mu\cdot\beta)$ can not be chosen arbitrary because of the anti-commutativity of the bracket and of the Jacobi identity.
We observe from  (\ref{crochetSVir2}) that if we put: $$e'_{\mu\cdot0}=e_{\mu\cdot0},~h'_{\alpha}=h_{\alpha}+\frac{\Psi_2(0,\mu\cdot\alpha)}{\mu\cdot\alpha}c_{\mu,2},~~ (\alpha\neq\overrightarrow{0}),$$
then we will have $$[e'_{\mu\cdot0},h'_{\alpha}]_{HVir_{\mu}}=(\mu\cdot\alpha)h'_{\alpha} \hbox{ for all }\alpha\in\mathbb{Z}^{n}.$$
This transformation is merely a change of basis and we can drop the prime and say that:
\begin{equation}\label{ident22}[e_{\mu\cdot0},h_{\alpha}]_{HVir_{\mu}}=(\mu\cdot\alpha)h_{\alpha}\hbox{ for all }\alpha\in\mathbb{Z}^{n}\end{equation}
From the Jacobi identity for $e_{\mu\cdot0},e_{\mu\cdot\alpha},h_{\beta}$, we get:
\begin{equation}\label{ident32}[e_{\mu\cdot0},[e_{\mu\cdot\alpha},h_{\beta}]_{HVir_{\mu}}]_{HVir_{\mu}}=\mu\cdot(\beta+\alpha)[e_{\mu\cdot\alpha},h_{\beta}]_{HVir_{\mu}}\end{equation}
Substituting (\ref{crochetSVir2}) in (\ref{ident32}) and using (\ref{ident22}) we get: $$\mu\cdot(\alpha+\beta)\Psi_2(\mu\cdot\alpha,\mu\cdot\beta)c_{\mu,2}=0.$$
But this is equivalent to $\alpha+\beta =\overrightarrow{0}$ or $\Psi_2(\mu\cdot\alpha,\mu\cdot\beta)=0$. Then $\Psi_2$ has the following form: \begin{equation} \label{exptheta2}\Psi_2(\mu\cdot\alpha,\mu\cdot\beta)=\delta_{\alpha,-\beta}\theta_2(\mu\cdot\alpha)\end{equation}
 where $\theta_2$ is a function from $\Gamma_{\mu}$ to $\mathbb{C}$.

We now work out  the  $2$-cocycle condition on $C_{\mu,2}$ for $(X_{\alpha,1},X_{\beta,1},X_{\gamma,2})=(e_{\mu\cdot\alpha},e_{\mu\cdot\beta},h_\gamma)$. If $\gamma+\beta+\alpha\neq \overrightarrow{0}$ then (\ref{CC1}) is satisfied. If $\gamma+\beta+\alpha=\overrightarrow{0}$, using (\ref{exptheta}) and the the fact that $\theta_2$ is odd, we get from (\ref{CC1}) the following equation:

$$(\mu\cdot\beta-\mu\cdot\alpha)\theta_2(\mu\cdot(\alpha+\beta))-(\mu\cdot(\alpha+\beta))\theta_2(\mu\cdot\beta)+(\mu\cdot\beta+\mu\cdot\alpha)\theta_2(\mu\cdot\alpha)=0$$
Put $x=\mu\cdot\alpha$ and $y=\mu\cdot\beta$, then we will obtain:
\begin{equation}\label{ef1}(y-x)\theta_2(x+y)=(y+x)(\theta_2(y)-\theta_2(x))\end{equation}
If $x=y$ or $x=-y$ the equation (\ref{ef1}) is satisfied.
If $x\neq y$ and $x\neq -y$, then (\ref{ef1}) is equivalent to the following equation:
\begin{equation}\label{ef2}\frac{\theta_2(x+y)}{x+y}=\frac{\theta_2(x)-\theta_2(y)}{x-y}\end{equation}

If $x\neq 0$, put $h(x)=\frac{\theta_2(2x)}{2x}$,~so we have:
\begin{equation}\label{ef3}\frac{\theta_2(x)-\theta_2(y)}{x-y}=h(\frac{x+y}{2})\end{equation}
This is the well known Acz\'{e}l functional equation (see \cite{Az}). Its general solution is given by:
 $$\theta_2(x)=ax^{2}+bx+c, \hbox{ for } a,b,c~\in~\mathbb{R}$$  and $h$ is $C^1$-function such that $h(x)=\theta_2'(x)$. But in our case $\theta_2(0)=0$ then $c=0$ and $\theta_2$ becomes:
$$\theta_2(x)=ax^{2}+bx~\forall~a,b~\in\mathbb{R}.$$

Following the choice of the $2$-cocycle in the twisted Heisenberg-Virasoro algebra corresponding to one variable ( $n=1$ ), we take  $a=1$ and $b=-1$ then we obtain: $$\theta_2(\mu\cdot\alpha)=(\mu\cdot\alpha)^{2}-\mu\cdot\alpha.$$

For the unicity of the $2$-cocycle $C_{\mu,3}$, we take $(j,k,l)=(1,2,2).$
Assume that there exists $\Psi_3:\Gamma_\mu\times\Gamma_\mu\rightarrow\in \mathbb{C}$ such that:
\begin{equation}\label{crochetSVir3}[h_{\alpha},h_{\beta}]_{HVir_{\mu}}=\Psi_3(\mu\cdot\alpha,\mu\cdot\beta)c_{\mu,3}.\end{equation}

From the Jacobi identity for $e_{\mu\cdot0},h_{\alpha},h_{\beta}$ we get:
$$\mu\cdot(\alpha+\beta)\Psi_3(\mu\cdot\alpha,\mu\cdot\beta)c_{\mu,3}=0.$$
But this is equivalent to $\alpha+\beta =\overrightarrow{0}$ or $\Psi_3(\mu\cdot\alpha,\mu\cdot\beta)=0$. Then $\Psi_3$ has the following form: \begin{equation} \label{exptheta3}\Psi_3(\mu\cdot\alpha,\mu\cdot\beta)=\delta_{\alpha,-\beta}\theta_3(\mu\cdot\alpha)\end{equation}
 where $\theta_3$ is a function from $\Gamma_{\mu}$ to $\mathbb{C}$.
We will have:
$$[h_{\alpha},h_{\beta}]_{HVir}=\delta_{\alpha,-\beta}\theta_3(\mu\cdot\alpha)c_{\mu,3},$$
$\hbox{with}~\theta_3(0)=0,~~\theta_3(-\mu\cdot\alpha)=-\theta_3(\mu\cdot\alpha).$

Let $\alpha,\beta,\gamma\in\mathbb{Z}^{n}$ and $\alpha+\beta+\gamma=\overrightarrow{0}.$
We apply the  $2$-cocycle condition for $(X_{\alpha,1},X_{\beta,2},X_{\gamma,2})=(e_{\mu\cdot\alpha},h_{\beta},h_\gamma)$
we obtain the equation:
\begin{equation}\label{equ3}(\mu\cdot\beta)\theta_3(\mu\cdot(\alpha+\beta))-(\mu\cdot\alpha+\mu\cdot\beta)\theta_3(\mu\cdot\beta)=0.\end{equation}
If we put $x=\mu\cdot\alpha$ and $y=\mu\cdot\beta$, then (\ref{equ3}) becomes:
\begin{equation}\label{ef4}y\theta_3(x+y)=(y+x)\theta_3(y)\end{equation}

If $x=0$ or $y=0$ the equation (\ref{ef4}) is satisfies.\\
If $x\neq 0$ and $y\neq 0$, then (\ref{ef4}) is equivalent to:
\begin{equation}\label{ef5}\frac{\theta_3(x+y)-\theta_3(y)}{x}=\frac{\theta_3(y)}{y}.\end{equation}
Let $X=x+y,Y=y$, then (\ref{ef5}) becomes:
\begin{equation}\label{ef6}\frac{\theta_3(Y)-\theta_3(X)}{Y-X}=\frac{\theta_3(Y)}{Y}.\end{equation}

If~$Y$ approaches $X$( $Y \rightarrow X)$ in the first member of (\ref{ef6}), we obtain the following differential equation:
$$\theta'_3(X)=\frac{\theta_3(X)}{X}$$ which has solution $\theta_3(X)= aX, a\in \mathbb{C}$.

Following the choice of the $2$-cocycle in the twisted Heisenberg-Virasoro algebra corresponding to one variable ( $n=1$ ), we take  $a=1/3$ then we obtain
$$\theta_3(\mu\cdot\alpha)=\frac{\mu\cdot\alpha}{3}.$$

\end{proof}

\begin{defi}
The central extension of $\mathbf{WA}(n)_{\mu}$ given by the three $2$-cocycles $C_{\mu,1},~C_{\mu,2}$ and $C_{\mu,3}$ in Theorem \ref{Ext2} is called the solenoidal Heisenberg-Virasoro algebra $(\mathbf{HVir}(n)_{\mu},[.,.]_{HVir_{\mu}})$ where
$${\mathbf{HVir(n)_{\mu}}}:=\mathbf{WA}(n)_{\mu}\oplus \mathbb{C}c_{\mu,1}\oplus\mathbb{C}c_{\mu,2}\oplus \mathbb{C}c_{\mu,3}.$$
and where its Lie bracket $[.,.]_{HVir_{\mu}}$ is generated by the following brackets:
\begin{equation}\label{HSvir0}
[e_{\mu\cdot\alpha},e_{\mu\cdot\beta}]_{HVir_{\mu}}=\mu\cdot(\beta-\alpha)e_{\mu\cdot(\alpha+\beta)}+\delta_{\alpha,-\beta}\frac{(\mu\cdot\alpha)^{3}-(\mu\cdot\alpha)}{12}c_{\mu,1}\end{equation}
\begin{equation}\label{HSvir2}[e_{\mu\cdot\alpha},h_\beta]_{HVir_{\mu}}=(\mu\cdot\beta)h_{\alpha+\beta}+\delta_{\alpha,-\beta}((\mu\cdot\alpha)^{2}-(\mu\cdot\alpha))c_{\mu,2}\end{equation}
\begin{equation}\label{HSvir1}[h_\alpha,h_\beta]_{HVir_{\mu}}=\delta_{\alpha,-\beta}\frac{(\mu\cdot\alpha)}{3}c_{\mu,3}\end{equation}
\begin{equation}\label{HSvir3}[c_{\mu,i},\mathbf{HVir}(n)_{\mu}]_{HVir_{\mu}}=0 \hbox{ for all } i=1,2,3.\end{equation}
\end{defi}
\begin{rmk}
\begin{itemize}
\item[1)]The name solenoidal Heisenberg-Virasoro algebra comes from the facts that $\mathbf{HVir}(n)_{\mu}$ contains a subalgebra isomorphic to $\mathbf{Vir}(n)_{\mu}$ generated by $\{e_{\mu\cdot\alpha},c_{\mu,1}\mid \alpha\in\mathbb{Z}^n\}$ and a subalgebra $$\mathbf{H}(n)_\mu:=\displaystyle\big(\oplus_{\alpha\in \mathbb{Z}^n}\mathbb{C}h_{\alpha}\big)\oplus\mathbb{C}c_{\mu,2}$$ which is isomorphic to an infinite dimensional Heisenberg algebra graded by $\mathbb{Z}^n$.\\
\item[2)]For  a given $2$-cocycle  $C_\mu:\mathbf{WA}(n)_{\mu}\times \mathbf{WA}(n)_{\mu}\rightarrow \mathbb{C}$,  there exists $(a_1,a_2,a_3)\in \mathbb{C}^3$ such that $C_\mu=a_1C_{\mu,1}+a_2C_{\mu,2}+a_3C_{\mu,3}$.
By bilinearity its expression is given as following:

$$\displaystyle \begin{array}{lll}C_\mu((e_{\mu\cdot\alpha},h_\beta),(e_{\mu\cdot\gamma},h_\eta))=&a_1C_{\mu,1}(e_{\mu\cdot\alpha},e_{\mu\cdot\gamma})+\\&a_2
(C_{\mu,2}(e_{\mu\cdot\alpha},h_{\eta})-C_{\mu,2}(e_{\mu\cdot\gamma},h_{\beta}))+\\&a_3C_{\mu,3}(h_{\beta},h_{\eta})\end{array}$$
for all $\alpha,\beta,\gamma,\eta\in \mathbb{Z}^n$.

Moreover,  The Lie bracket of $\mathbf{HVir}(n)_{\mu}$ is given by:
$$[X,Y]_{HVir_{\mu}}= [X,Y]+ C_{\mu}(X,Y), \hbox{ for all } X,Y\in \mathbf{HVir}(n)_{\mu}.$$
\end{itemize}
\end{rmk}
\section{ Harish Chandra modules for ${\mathbf{HVir}(n)_\mu}$}
\subsection{Generalities on Harish-Chandra modules}
Let $V$ be a nonzero  ${\mathbf{HVir}(n)_\mu}$-module. Suppose that the central elements $c_{\mu,1}, c_{\mu,2} , c_{\mu,3} \hbox{ and } h_0$ act as scalars $c_1, c_2 , c_3 , F$ respectively, on $V$. Set
$$V_{\lambda} = \{v \in V |d_{\mu}v = \lambda v\},$$ which is called a weight space of weight $\lambda$. Then $V$ is
called a weight module if $V =\oplus_{\lambda\in\mathbb{C}}V_{\lambda}$. Denote $supp(V) = \{\lambda|V_{\lambda} \neq 0\}$,
which is called the support of $V.$
\begin{defi}
A weight $\mathbf{HVir}(n)_{\mu}$-module $V$  is called Harish-Chandra if dim $V_\lambda <\infty$ for all $\lambda\in supp(V )$ and is called uniformly
bounded or cuspidal if there is some $N \in \mathbb{N}$ such that $dim V_\lambda< N$ for all $\lambda\in supp(V )$.
\end{defi}
\begin{defi}
A weight ${\mathbf{HVir}(n)_\mu}$-module $V$ is called a module of the intermediate series if it is
indecomposable and all its weight spaces are at most one dimensional.
\end{defi}
\subsection{Intermediate series of $\mathbf{HVir}(n)_\mu$}
\begin{pro}
Let $T_{\mu}(a,b,F)$ the $\Gamma_\mu$-graded vector space:
$$T_{\mu}(a,b,F) =\oplus_{\mu\cdot\kappa\in\Gamma_\mu}v_{\mu\cdot\kappa+a}$$ where $a,b,F\in \mathbb{C}$. We define an action of $\mathbf{HVir}(n)_\mu$ on  $T_{\mu}(a,b,F)$ by:
\begin{equation} \label{seInt}\begin{array}{ccc}
e_{\mu\cdot\alpha}.v_{\mu\cdot\kappa+a}=(a+\mu\cdot\kappa+b(\mu\cdot\alpha))v_{\mu\cdot(\kappa+\alpha)+a},\\
h_{\alpha}.v_{\mu\cdot\kappa+a}=Fv_{\mu\cdot(\kappa+\alpha)+a},\\
c_{\mu,1}v_{\mu\cdot\kappa+a}=0,c_{\mu,2}v_{\mu\cdot\kappa+a}=0,c_{\mu,3}v_{\mu\cdot\kappa+a}=0\end{array}\end{equation}
for all $\kappa,\alpha\in\mathbb{Z}^{n}$.
Then $T_{\mu}(a,b,F)$ is a $\mathbf{HVir}(n)_\mu$-module for this action.
\end{pro}
\begin{rmk}
The weight spaces of $T_{\mu}(a,b,F)$ are one dimensional. Then $T_{\mu}(a,b,F)$ are called cuspidal or intermediate series modules.
\end{rmk}
It is easy to check that the $\mathbf{HVir}(n)_\mu$-module $T_{\mu}(a, b, F)$ is reducible if and only if $F = 0,
a \in \Gamma_\mu \hbox{ and } b = 0, 1.$ The module $T_{\mu}(0, 0, 0)$ contains $\mathbb{C}v_0$ as a submodule and the quotient $T_{\mu}(0, 0, 0)/\mathbb{C}v_0$ is irreducible. The module $T_{\mu}(0, 1, 0)$ contains $\oplus_{\alpha\in \mathbb{Z}^n\setminus\{\overrightarrow{0}\}} \mathbb{C} v_{\mu\cdot\alpha}$ as
irreducible submodule of codimension one. By duality, it will be isomorphic to $T_{\mu}(0, 0, 0)/\mathbb{C}v_0$.  We will denote it $\overline{T}_{\mu}(0,0,0).$

Let V be a nontrivial irreducible weight $\mathbf{HVir}(n)_\mu$-module with weight multiplicity one.
We may assume that $h_0, c_{\mu,1}, c_{\mu,2}, c_{\mu,3}$ act as scalars $F,c_{1}, c_{2}, c_{3}$ respectively.

Following Lemma 3.1 and Lemma 3.2 in \cite{LuZhao}, we will prove the following proposition:
\begin{pro} \label{IntHVirn} Let $\mu=(\mu_1,\ldots,\mu_n)\in \mathbb{C}^n$ a generic element that is: $$\mu\cdot\alpha\neq 0,~\forall~\alpha=(\alpha_1,\ldots,\alpha_n)\in \mathbb{Z}^n\setminus\{\overrightarrow{0}\}.$$
Let  $V:=\displaystyle\oplus_{\mu\cdot\kappa\in \Gamma_\mu}\mathbb{C}v_{\mu\cdot\kappa}$ be a  $\mathbf{HVir}(n)_\mu$-module the action given by:
$$e_{\mu\cdot\alpha}.v_{\mu\cdot\kappa+a}=(a+\mu\cdot\kappa+b(\mu\cdot\alpha))v_{\mu\cdot(\kappa+\alpha)+a}.$$
$$h_{\alpha}.v_{\mu\cdot\kappa+a}=F_{\mu\cdot\alpha,\mu\cdot\kappa}v_{\mu\cdot(\kappa+\alpha)+a}\hbox{ and }c_{\mu,i}v_{\mu\cdot\kappa+a}=c_iv_{\mu\cdot\kappa+a}\hbox{ for } i\in\{1,2,3\}.$$

Then all $F_{\mu\cdot\alpha,\mu\cdot\kappa}$ are equal to a constant $F$ and $c_i=0$ for $i\in\{1,2,3\}$
and such module  $V$ is isomorphic to $T_{\mu}(a, b, F)$.
\end{pro}

\begin{proof} It is strait forward to prove that $c_1 = 0$ by restriction to $\mathbf{Vir}(n)_\mu$ and using results by \cite{ChPr,MP}.

 It is clear that $supp(V)\subset a + \Gamma_\mu \hbox{ for some } a \in \mathbb{C}$.
 We give a proof by induction on $n$ to prove that  $F_{\mu\cdot\alpha,\mu\cdot\kappa}=F$ for all $\alpha,\kappa\in \mathbb{Z}^n$.

For $n=1$, Proposition \ref{IntHVirn} is Lemma 3.1 in the paper \cite{LuZhao}.\\
Let us prove the case of $n=2$. Let $h_{(l,m)}=t_1^lt_2^m$ and let $\mathcal{H}(2)=\oplus_{(l,m)\in \mathbb{Z}^2}\mathbb{C}h_{(l,m)}\oplus \mathbb{C}c_{\mu,3}$ be the Heisenberg subalgebra of $\mathbf{HVir}(2)_\mu$ and let $V=\oplus_{(p,q)\in \mathbb{Z}^2}\mathbb{C}v_{\mu_1p+\mu_2q}$. Let us fix $l$ and $p$ and consider  $\mathcal{H}_l(1)=\oplus_{m\in \mathbb{Z}}\mathbb{C}h_{(l,m)}\oplus \mathbb{C}c_{\mu,3}$ and $V_p=\oplus_{q\in \mathbb{Z}}\mathbb{C}v_{\mu_1p+\mu_2q}$. Then  $\mathcal{H}_l(1)$ is a subalgebra of $\mathcal{H}(2)$ isomorphic to the  Heisenberg algebra $\mathcal{H}(1)$ and $V_p$ is an intermediate module for $\mathcal{H}_l(1)$. By  Lemma 3.1 in the paper \cite{LuZhao},
$h_{(l,m)}$ acts by a constant $F_{l,p}$ which depends on $l,p\in \mathbb{Z}$ but independent of $m$ and $q$ and $c_{\mu,3}$ act by zero on $V_p$ for all $p$. If we interchange $n$ by $m$ and $p$ by $q$, then $F_{l,p}$ will be independent of $l$ and $p$ and then it will be a constant $F$ for all $(l,m)\in \mathbb{Z}^2$   and $c_{\mu,3}$ act by zero on all $V$.

Now, assume that the proposition is true on $\mathbb{Z}^{n-1}$ where $n\in \mathbb{N}$ and $n\geq 2$. Let  $\mathcal{H}_m(n-1)=\oplus_{\alpha\in \mathbb{Z}^{n-1}}\mathbb{C}h_{(\alpha,m)}\oplus \mathbb{C}c_{\mu,3}$ and $V_q=\oplus_{\beta\in \mathbb{Z}^{n-1}}\mathbb{C}v_{\mu'\cdot\beta+\mu_nq}$ where $\mu'=(\mu_1,\ldots,\mu_{n-1})$. By the induction hypothesis  $\mathcal{H}_m(n-1)$ acts by a constant $F_{m,q}$ on $V_q$ which depends only on $m$ and $q$ for the moment and $c_{\mu,3}$ act by zero on $V_q$.
 Now if we fix $\alpha,\beta\in \mathbb{Z}^{n-1}$ and consider $\mathcal{H}_{\alpha}(1) =\oplus_{m\in \mathbb{Z}}\mathbb{C}h_{(\alpha,m)}\oplus \mathbb{C}c_{\mu,3}$
 and $V_\beta:= \oplus_{q\in \mathbb{Z}}\mathbb{C}v_{\mu'\cdot\beta+\mu_nq}$, then $F_{m,q}$ will be independent of $m$ and $q$ and then it will be a constant $F$ for all $(\alpha,m)\in \mathbb{Z}^n$.
\end{proof}
\subsection{Generalized highest weight modules}
For $\mu=(\mu_1,\mu_2,\ldots,\mu_n)\in \mathbb{C}^n$, let $\mu'=(\mu_2,\ldots,\mu_n)\in \mathbb{C}^{n-1}$. For any $\alpha=(\alpha_1,\ldots,\alpha_n)\in \mathbb{Z}^{n}$ we have $\mu\cdot\alpha=\mu_1\alpha_1+\mu'\cdot\alpha'$  where $\alpha'=(\alpha_2,\ldots,\alpha_n)$. This induces a natural embedding of $\Gamma_{\mu'}$ in $\Gamma_{\mu}$
given by $\mu'\cdot\alpha'\mapsto \mu\cdot(0,\alpha')$. The embedding $\Gamma_{\mu'}\hookrightarrow \Gamma_{\mu}$ as defined below, induces an embedding of the Lie algebra $\mathbf{HVir}(n-1)_{\mu'}$ into the Lie algebra $\mathbf{HVir}(n)_{\mu}$ given by:
 $$e_{\mu'\cdot\alpha'}\mapsto e_{\mu\cdot(0,\alpha')}  \hbox{ and }h_{\alpha'}\mapsto h_{(0,\alpha')} .$$

Let $A_{n-1}=\mathbb{C}[t_2^{\pm1},\ldots,t_{n}^{\pm1}]$, then we have the following $\mathbb{Z}$-grading of $\mathbf{HVir}(n)_{\mu}$:
$$\mathbf{HVir}(n)_{\mu}=\oplus_{i\in\mathbb{Z}}\mathbf{HVir}(n)_{\mu}^{i}$$ where $\mathbf{HVir}(n)_{\mu}^{0}=A_{n-1}d_{\mu}\oplus A_{n-1} \oplus \sum_{i=1}^{3}\mathbb{C}c_{\mu,i}$ and $\mathbf{HVir}(n)_{\mu}^{i}=t_1^{i}A_{n-1}d_{\mu}\oplus t_1^{i}A_{n-1}\oplus\sum_{i=1}^{3}\mathbb{C}c_{\mu,i}$ if $i\neq 0$.
The Lie subalgebra $\mathbf{HVir}(n)_{\mu}^{0}$ of $\mathbf{HVir}(n)_{\mu}$ is isomorphic to $\mathbf{HVir}(n-1)_{\mu'}$.
The algebra $\mathbf{HVir}(n)_{\mu}$ has a triangular decomposition $$\mathbf{HVir}(n)_{\mu}^{+}\oplus\mathbf{HVir}(n)_{\mu}^{0}\oplus\mathbf{HVir}(n)_{\mu}^-$$
where $\mathbf{HVir}(n)_{\mu}^{\pm}:=\oplus_{i\in\pm\mathbb{N}}\mathbf{HVir}(n)_{\mu}^{i}$.

For $a,b\in\mathbb{C}$, we denote $T_{\mu'}(a,b,F)$ the $\mathbf{HVir}(n)_{\mu}^{0}$ module of tensor fields $$T_{\mu'}(a,b,F)=\oplus_{\mu'\cdot\kappa'\in\Gamma_{\mu'}}\mathbb{C}v_{\mu'\cdot\kappa'}$$
subject to the action:
\begin{equation}\label{act1}\begin{array}{lll}e_{\mu'\cdot\alpha'}.v_{\mu'\cdot\kappa'}=(a+\mu'\cdot\kappa'+b(\mu'\cdot\alpha'))v_{\mu'.(\alpha'+\kappa')}, \\ h_{\alpha'}.v_{\mu'.\kappa'}=Fv_{\mu'\cdot(\kappa'+\alpha')},\\ c_{\mu,i}.v_{\mu'\cdot\kappa'}=0 \hbox{ for } i=1,2,3\hbox{ and } \mu'\cdot\kappa',\mu'\cdot\alpha'\in \Gamma_\mu'\cdot\end{array}\end{equation}

We extend the $\mathbf{HVir}(n)_{\mu}^{0}$ module structure on $T_{\mu'}(a,b,F)$ given by (\ref{act1}) to $\mathbf{HVir}(n)_{\mu}^{+}\oplus\mathbf{HVir}(n)_{\mu}^{0}$
where the elements of $\mathbf{HVir}(n)_{\mu}^{+}$ act by zero on $T_{\mu'}(a,b,F)$. Let
$$\widetilde{M}(a,b,\Gamma_{\mu'})=Ind^{\mathbf{HVir}(n)_{\mu}}_{\mathbf{HVir}(n)_{\mu}^{+}\oplus\mathbf{HVir}(n)_{\mu}^{0}}T_{\mu'}(a,b,F)$$
be the generalized Verma module. As vector spaces we have $\widetilde{M}(a,b,\Gamma_{\mu'})\cong U(\mathbf{HVir}(n)_{\mu}^{-})\otimes_{\mathbb{c}}T_{\mu'}(a,b,F).$
 The module $\widetilde{M}(a,b,\Gamma_{\mu'})$ has a unique maximal proper
submodule $\overline{M}(a,b,\Gamma_{\mu'})$ trivially intersecting  $T_{\mu'}(a,b,F)$. The quotient module
$$L(a,b,\Gamma_{\mu'}):=\widetilde{M}(a,b,\Gamma_{\mu'})/\overline{M}(a,b,\Gamma_{\mu'})$$
 is uniquely determined by the constants $a,b$ and $$L(a,b,\Gamma_{\mu'})=\oplus_{i>0} L_{a-i\mu_1+\Gamma_{\mu'}}$$
where
$L_{a-i\mu_1+\Gamma_{\mu'}} = \oplus_{\mu'\cdot\kappa\in\Gamma_{\mu'}}L_{a-i\mu_1+\mu'\cdot\kappa}$
and $$L_{a-i\mu_1+\mu'\cdot\kappa}=\{v\in L/d_{\mu}v=(a-i\mu_1+\mu'\cdot\kappa)v\}$$
We can similarly define $\widetilde{M}_{a+i\mu_1+\Gamma_{\mu'}}$ and $\widetilde{M}_{a-i\mu_1+\Gamma_{\mu'}}$.
\begin{defi}
 Let $(u_1,\ldots,u_n)$ be a $\mathbb{Z}$-basis of $\Gamma_\mu$ and let $\Gamma_\mu^{>0}:= \mathbb{Z}^+u_1\oplus\ldots\oplus\mathbb{Z}^+u_n$ and $\mathbf{HVir}(n)_{\mu}^{>0}:= \oplus_{u\in \Gamma_\mu^{>0}} (\mathbf{HVir}(n)_{\mu})_u.$ Let $V$ be a weight module such that there exists $\lambda_0\in Supp(V)$ and a nonzero vector $v_{\lambda_0}\in V_{\lambda_0}$ such that : $\mathbf{HVir}(n)_{\mu}^{>0}v_{\lambda_0}=0$. Then $V$ is said to be a generalized highest weight module
with generalized highest weight $\lambda_0$ and generalized highest weight vector $v_{\lambda_0}$. Such module $V$ is denoted by $V(\lambda_0)$.
\end{defi}
In G.Liu and X.Guo (see \cite{LiuGuo} Theorem.16)  , it is proved that for a generalized Heisenberg-Virasoro algebra an irreducible weight module with finite dimensional weight spaces is either a cuspidal or a generalized highest weight module. In our particular case, any irreducible $\mathbf{HVir}(n)_{\mu}$-module is either cuspidal or isomorphic to $L(a,b,\Gamma_{\mu'})$.
\begin{defi}
A  $\mathbf{HVir}(n)_{\mu}$-module $V$ is called a \textbf{dense} module if $supp(V)=a+\Gamma_\mu,~a\in \mathbb{C}$ and is called a \textbf{cut} module if $supp(V)\subset \lambda+\gamma+\Gamma^{(\alpha)}_{\leq 0}$ where $\Gamma^{(\alpha)}_{\leq 0}:=\{\mu\cdot\beta|\beta\in \mathbb{Z}^n \hbox{ and }\beta.\alpha\leq 0\} \hbox { and } \gamma\in\Gamma_\mu$.
\end{defi}
The modules $T_{\mu}(a,b,F)$ are irreducible \textbf{dense} modules and $L(a,b,\Gamma_{\mu'})$ are irreducible \textbf{cut} modules.

The following theorem is a consequence of Theorem 15 and Theorem 16 in \cite{LiuGuo}. It classifies Harish-chandra modules of $\mathbf{HVir}(n)_{\mu}$.

\begin{thm} \label{ClssThm}
Let  $V$ be a nontrivial irreducible weight module with finite dimensional
weight spaces over the Heisenberg solenoidal-Virasoro algebra $\mathbf{HVir}(n)_{\mu}$.
\begin{itemize}
\item[1)] If $n=1$ then $\Gamma_{\mu}=\mu\mathbb{Z}\simeq\mathbb{Z}$, then $V$ is  of intermediate series or highest or lowest module (see \cite{LiuJiang,LuZhao}).
\item[2)] If $n\geq 2$, then $V$ is isomorphic to one of the following modules:
\begin{itemize}
\item[a)]$V\cong T_{\mu}(a,b,F)$ for $(a,b)\in \mathbb{C}^2\setminus\{(0,0)\}$ or $V\cong \overline{T}_{\mu}(0,0,0)$.
\item[b)] $V\cong L(a,b,\Gamma_{\mu'})$ for some $a,b\in \mathbb{C}$.
\end{itemize}

\end{itemize}
\end{thm}

\section{Simple Weight ${\mathbf{HVir}(n)_\mu}$-modules having infinite dimensional weight spaces}
Let $\displaystyle\mathbb{Z}^{n}$ be the free abelian group of rank $n$ whose elements are sequences of $n$ integers, and operation is the addition. A group order on $\displaystyle \mathbb{Z}^{n}$ is a total order, which is compatible with addition, that is
$$ a<b\quad {\text{ if and only if }}\quad a+c<b+c.$$
The lexicographical order $<_{lex}$ is a group order on $\mathbb{Z}^{n}.$

We transport the lexicographic order $<_{lex}$ on $\mathbb{Z}^{n}$ to $\Gamma_\mu$ that is
$$\mu\cdot\alpha\prec \mu\cdot\beta \hbox{ if and only if } \alpha<_{lex} \beta.$$

Let us introduce
$$\Delta^{+}:=\{\alpha\in\mathbb{Z}^{n}|\overrightarrow{0}<_{lex}\alpha\}~~,~~\Delta^{-}:=\{\alpha\in\mathbb{Z}^{n}|\alpha<_{lex}\overrightarrow{0}\}$$
$$\Gamma_\mu^+:=\sigma_\mu(\Delta^+):=\{\mu\cdot\alpha| \overrightarrow{0}<_{lex}\alpha\}~~,~~\Gamma_\mu^-:=\sigma_\mu(\Delta^-):=\{\mu\cdot\alpha| \alpha<_{lex}\overrightarrow{0}\}$$

Let  $(\textbf{Vir}(n)_{\mu})_{+},(\textbf{Vir}(n)_{\mu})_{-},(\textbf{Vir}(n)_{\mu})_{0}, (\textbf{H}(n)_{\mu})_{+},(\textbf{H}(n)_{\mu})_{-}$ and $(\textbf{H}(n)_{\mu})_{0}$ be the subalgebras defined by:
  $$\begin{array}{cc}(\textbf{Vir}(n)_{\mu})_{+}=\displaystyle\bigoplus_{\alpha\in\Delta^{+}}\mathbb{C}e_{\mu\cdot\alpha}
,~~(\textbf{Vir}(n)_{\mu})_{-}=\displaystyle\bigoplus_{\alpha\in\Delta^{-}}\mathbb{C}e_{\mu\cdot\alpha}, (\textbf{Vir}(n)_{\mu})_{0}=\mathbb{C}d_{\mu}\oplus\mathbb{C}c_{\mu,1},\\
(\textbf{H}(n)_{\mu})_{+}=\displaystyle\bigoplus_{\alpha\in\Delta^{+}}\mathbb{C}h_{\alpha}
,~~(\textbf{H}(n)_{\mu})_{-}=\displaystyle\bigoplus_{\alpha\in\Delta^{-}}\mathbb{C}h_{\alpha},(\textbf{H}(n)_{\mu})_{0}=\mathbb{C}h_{0}\oplus\mathbb{C}c_{\mu,2}\end{array}.$$

The algebra ${\mathbf{HVir}(n)_{\mu}}$ has the following triangular decomposition: $$\mathbf{HVir}(n)_{\mu} = (\mathbf{HVir}(n)_{\mu})_{+}\oplus(\mathbf{HVir}(n)_{\mu})_{0}\oplus (\mathbf{HVir}(n)_{\mu})_{-}$$ where,
$$(\mathbf{HVir}(n)_{\mu})_{\pm}=(\textbf{H}(n)_{\mu})_{\pm}\oplus(\textbf{Vir}(n)_{\mu})_{\pm},$$
$$(\mathbf{HVir}(n)_{\mu})_{0}=\mathbb{C}e_{\mu\cdot0}\oplus\mathbb{C}h_0\oplus\mathbb{C}c_{\mu,1}\oplus\mathbb{C}c_{\mu,2}\oplus\mathbb{C}c_{\mu,3}.$$

Let  $\lambda=(\lambda_\mu,~c_0~,c_1 ,c_2,c_3)\in \mathbb{C}^5$ and denote $\mathbf{HB}(n)_+:=(\mathbf{HVir}(n)_{\mu})_{0}\oplus(\mathbf{HVir}(n)_{\mu})_{+}$. Let the one dimentional $\mathbf{HB}(n)_+$-module $\mathbb{C}_\lambda$ where the action is given by:
$$e_{\mu\cdot0}.1_\lambda = \lambda_\mu 1_\lambda, h_{0}.1_\lambda =c_01_\lambda, c_{\mu,1}.1_\lambda= c_11_\lambda, c_{\mu,2}.1_\lambda = c_21_\lambda, c_{\mu,3}.1_\lambda = c_31_\lambda.$$

The Verma module of $\mathbf{HVir}(n)_{\mu}$ is the induced weight module: $$M(\lambda)=Ind^{\mathbf{HVir}(n)_{\mu}}_{\mathbf{HB}(n)_+}\mathbb{C}_\lambda:=U(\mathbf{HVir}(n)_{\mu})\otimes_{U(\mathbf{HB}(n)_+)}\mathbb{C}_\lambda$$

 The Verma module $M(\lambda)$ has a maximal proper submodule $\widetilde{M(\lambda)}$ and the quotient $V(\lambda):=M(\lambda)/\widetilde{M(\lambda)}$ will be irreducible and called the irreducible highest module with highest weight $\lambda$. Moreover, every irreducible highest module will be constructed with this manner.

The irreducible lowest weight modules $V(\lambda)^{\vee}$ of lowest weight $\lambda$ are constructed in the same manner of the ones in the case of the $\mathbf{HVir}(n)_\mu$ algebra.

We can also consider the Verma module of $\mathbf{Vir}(n)_\mu$: $$K(\nu):=Ind^{\mathbf{Vir}(n)_{\mu}}_{(\mathbf{Vir}(n)_\mu)_0\oplus(\mathbf{Vir}(n)_\mu)_+}\mathbb{C}_\nu$$ where
 $\nu=(\lambda_{\mu},c_1),e_{\mu\cdot0}.1_\nu= \lambda_\mu 1_\nu, c_{\mu,1}.1_\nu= c_11_\nu$ and $(\mathbf{Vir}(n)_\mu)_+$ acts by $0$.

 The module $K(\nu)$ has a maximal proper submodule $\widetilde{K(\nu)}$ and the quotient $L(\nu):=K(\nu)/\widetilde{K(\nu)}$ is an
 irreducible highest $\mathbf{Vir}(n)_\mu$-module.

The algebra ${\mathbf{HVir}(n)_{\mu}}$ has also  the following generalized triangular decomposition: $$\mathbf{HVir}(n)_{\mu} = (\mathbf{H}(n)_{\mu})_{-}\oplus\mathbf{Vir}(n)_{\mu}\oplus(\mathbf{H}(n)_{\mu})_{0}\oplus\mathbb{C} c_{\mu,3}\oplus(\mathbf{H}(n)_{\mu})_{+}.$$

Let $(\mathbf{P}(n)_{\mu})_+:=\mathbf{Vir}(n)_{\mu}\oplus(\textbf{H}(n)_{\mu})_{0}\oplus(\mathbf{H}(n)_{\mu})_{+}\oplus\mathbb{C}c_{\mu,3}$.
Let $L(\nu)$ be an irreducible $\mathbf{Vir}(n)_{\mu}$-module. Extend it to $(\mathbf{P}(n)_{\mu})_+$-module by letting $h_0$ acts by $c_0$, $c_{\mu,2}$ acts by $c_2$, $c_{\mu,3}$ acts by $c_3$ and $(\mathbf{H}(n)_{\mu})_{+}$ acts by $0$.
Let the generalized Verma module of ${\mathbf{HVir}(n)_{\mu}}$:
$$G(\lambda)=Ind^{\mathbf{HVir}(n)_{\mu}}_{(\mathbf{P}(n)_{\mu})_+}L(\nu)$$ where $\lambda=(\nu, c_0,c_2,c_3).$
 The module $G(\lambda)$ has a maximal submodule $\widetilde{G(\lambda)}$ and the quotient is irreducible module $V(\lambda)$.
 As a module of $\mathbf{Vir}(n)_{\mu}$ it contains $L(\nu)$ as a submodule.

\begin{thm} Let $V(\lambda)$  be the  irreducible highest weight module of $\mathbf{HVir}(n)_\mu$, then there exists $\alpha\in supp(V(\lambda))$ such that $V(\lambda)_\alpha$ is an infinite dimensional  weight subspace of $V(\lambda)$.

We have the same assertion for the lowest weight module $V(\lambda)^{\vee}$.
\end{thm}
\begin{proof}
As a module of  $\mathbf{Vir}(n)_\mu$, $V(\lambda)$ contains $L(\nu)$ as submodule. Using results in \cite{AM}, $L(\nu)$ has infinite dimensional  weight subspaces. We deduce that $V(\lambda)$ has submodules of infinite dimensional weight spaces.
\end{proof}

\end{document}